\documentclass{article}
\usepackage{graphicx} 
\usepackage{fullpage}
\usepackage{amssymb, amsmath}
\usepackage{amsthm} 
\usepackage{hyperref}
\usepackage{color}

\usepackage{tikz}
\usetikzlibrary{arrows.meta,positioning}

\newtheorem{theorem}{Theorem}
\newtheorem{lemma}[theorem]{Lemma}

\theoremstyle{definition}

\newtheorem*{theorem*}{Theorem}
\newtheorem*{lemma*}{Lemma}

\newtheorem{claim}{Claim}
\newenvironment{proofc}{\begin{proof}[Proof of Claim]}{\end{proof}}

\newcommand\ex{\ensuremath{\mathrm{ex}}}
\newcommand\EX{\ensuremath{\mathrm{EX}}}

\title{On the Tur\'an number of the directed path}
\author{Daniel P. Johnston\thanks{Department of Mathematics, Trinity College, Hartford, CT. E-mail: \texttt{daniel.johnston@trincoll.edu}.} \qquad Cory Palmer\thanks{Department of Mathematical Sciences, University of Montana, Missoula, MT. E-mail: \texttt{cory.palmer@umontana.edu}. Supported by NSF grant DMS-2503179 and a grant from the Simons Foundation [SFI-MPS-TSM-00013277, CP].} \qquad Amites Sarkar\thanks{Department of Mathematics, Western Washington University, Bellingham, WA. E-mail: \texttt{sarkara@wwu.edu}.}}

\begin{document}

\maketitle

\begin{abstract}
    In this note we determine the maximum number of arcs in a digraph on $n$ vertices that does not contain a length-$k$ directed path $\overrightarrow{P}_{k+1}$ for $n \geq 50k^6$. This improves a theorem of Zhou and Li [\emph{Graphs
Combin.} 39(3), 2023] who proved the result with a superexponential threshold on $n$.
\end{abstract}

\section{Introduction}

One of the oldest results in extremal graph theory is the Erd\H os-Gallai theorem~\cite{EG} on the extremal number for paths. Erd\H os and Gallai proved, in 1959, that ${\rm ex}(n,P_{k+1})$ (the maximum number of edges in an undirected graph $G$ on $n$ vertices containing no path $P_{k+1}$ on $k+1$ vertices) is at most $\tfrac12(k-1)n$, with equality if and only if $n=kt$, and $G$ is the disjoint union of $t$ copies of $K_k$. This theorem was extended in 1975 by Faudree and Schelp~\cite{FS}, who determined ${\rm ex}(n,P_{k+1})$, together with all the extremal graphs, for all $n$ and $k$.

\medskip

It's natural to ask the same question for {\it oriented} graphs, in which the path $P_{k+1}$ is replaced by an oriented path $\overrightarrow{P}_{k+1}$. This somewhat similar question can be easily answered by applying two even older theorems, and the extremal construction is completely different: it's the transitive T\'uran digraph $\overrightarrow{T}_{n,k}$. To define $\overrightarrow{T}_{n,k}$, we start with the $k$-partite T\'uran graph $T_{n,k}$ on $n$ vertices, and an arbitrary ordering of the parts $V_1,\ldots,V_k$. We then direct an edge between $V_i$ and $V_j$, where $i<j$, so that it points from $V_i$ to $V_j$ (see Figure~\ref{turan-fig}). Strictly speaking, when $k \nmid n$, this creates a class of oriented graphs, but, since they all have the same number of arcs, we will denote any one of them with the symbol $\overrightarrow{T}_{n,k}$. Now for the proof. Suppose an oriented graph ${G}$ contains more arcs than $\overrightarrow{T}_{n,k}$. By T\'uran's theorem, ${G}$ contains a tournament, i.e., a complete oriented graph, on $k+1$ vertices. But by R\'edei's theorem~\cite{Redei}, this tournament contains a directed Hamilton path $\overrightarrow{P}_{k+1}$.

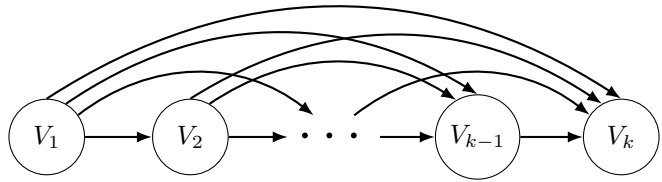
\begin{figure}[ht]
\centering
\begin{tikzpicture}[
    scale=0.95,
    v/.style={circle,draw,fill=white,minimum size=2pt,inner sep=0pt},
    cls/.style={circle,draw,minimum size=1cm},
    arr/.style={-{Latex[length=2mm]},thick},
    lab/.style={font=\small}
]

    \node[cls] (V0) at (0,0) {$V_1$};
    \node[cls] (V1) at (2,0) {$V_2$};
    \node[lab] (mid) at (4,0) {\huge $\cdots$};
    \node[cls] (Vk1) at (6,0) {$V_{k-1}$};
    \node[cls] (Vk) at (8,0) {$V_{k}$};

    \draw[arr] (V0) -- (V1);
    \draw[arr] (V1) -- (mid);
    \draw[arr] (mid) -- (Vk1);
    \draw[arr] (Vk1) -- (Vk);
    \draw[arr,bend left=33] (V0.north) to (Vk.north);
    \draw[arr,bend left=33] (V0.60) to (Vk1.north);
    \draw[arr,bend left=38] (V0.35) to ([xshift=-8pt]mid.north);
    \draw[arr,bend left=33] (V1.north) to (Vk.120);
    \draw[arr,bend left=33] (V1.60) to (Vk1.120);
    \draw[arr,bend left=38] ([xshift=8pt]mid.north) to (Vk.145);

\end{tikzpicture}
\caption{A transitive Tur\'an digraph.}\label{turan-fig}
\end{figure}

The next natural question is for {\it digraphs}, which can have arcs pointing in both directions between two vertices. Here, the extremal digraph is again $\overrightarrow{T}_{n,k}$, but there is no known short proof, and we also need $n$ to be sufficiently large. More precisely, in 2023, Zhou and Li~\cite{zhli} proved the following theorem. Here and throughout $a(G)$ denotes the number of arcs in $G$ and $\ex(n,\overrightarrow{F})$ is the Tur\'an number of the digraph $\overrightarrow{F}$, i.e., the maximum number of arcs in a $\overrightarrow{F}$-free digraph on $n$ vertices.

\begin{theorem}[Zhou--Li~\cite{zhli}]
    If $n \geq n_0(k)$, then
    \[
    \ex(n,\overrightarrow{P}_{k+1}) = a(\overrightarrow{T}_{n,k}).
    \]
    Moreover, the extremal digraphs are precisely the
 transitive Tur\'an digraphs $\overrightarrow{T}_{n,k}$.
\end{theorem}

Zhou and Li's proof uses the Erd\H os-Stone theorem~\cite{ES}, which leads to a superexponential bound $n_0(k)$. They conjectured, by contrast, that the theorem should hold for $n\ge 3k$. For $k+1\le n\le 2k-1$, $\overrightarrow{T}_{n,k}$ is not extremal, and for $2k\le n\le 3k-2$, if the extremal digraph is a $\overrightarrow{T}_{n,k}$, it is not unique---see~\cite{zhli}. In this note we present a new proof of this theorem with a polynomial threshold $n_0(k)$.

\begin{theorem}\label{main-thm}
If $n\ge 50k^6$, then \[{\rm ex}(n,\overrightarrow{P}_{k+1}) = 
a(\overrightarrow{T}_{n,k}).
\]
Moreover, the extremal digraphs are precisely the
 transitive Tur\'an digraphs $\overrightarrow{T}_{n,k}$.
\end{theorem}

It's natural to try to prove this theorem by induction on $n$, and the induction step does indeed go through. But the base case is trickier, since the conclusion of the theorem isn't even true when $n\le 3k-2$. In some sense, the theorem becomes ``more" true as $n$ gets larger. To deal with this situation, which frequently occurs in extremal graph theory, we will use Simonovits' method of {\it progressive induction}~\cite{Miki,Survey, Sprangel}. The idea is as follows. First, we establish a weak form of the base case (in Lemma~\ref{bad-bound} below). Then, to compensate for a weaker base case, we prove a stronger induction step. Namely, we show, in Lemma~\ref{main-lemma} below, that if $n$ is sufficiently large, and if $G$ is an extremal digraph on $n$ vertices that has $f$ more arcs than $\overrightarrow{T}_{n,k}$, then we can find a digraph $G'$ on $n-k$ vertices that has at least $f+1$ more arcs than $\overrightarrow{T}_{n-k,k}$. In other words, if there is a counterexample to the theorem, then there is a {\it strictly worse} (i.e., more extreme) counterexample on fewer vertices. Repeating this step enough times, we eventually contradict even our weaker base case, and we are done.

\section{Proof of Theorem~\ref{main-thm}}

First we prove our weak base case. Just for fun, we will apply the following theorem, which was proved independently by four authors in the 1960s (their papers are in English, French, German and Russian).

\begin{theorem}[Gallai--Hasse--Roy--Vitaver \cite{MR233733, MR179105,MR225683,MR145509}]\label{GHRV}
    If an oriented graph $G$ is $\overrightarrow{P}_{k+1}$-free, then its underlying graph  is $k$-colorable.
\end{theorem}

We will also need the classical result of Erd\H os and Gallai quoted in the introduction.

\begin{theorem}[Erd\H os--Gallai \cite{EG}]\label{EG}
    If $G$ is a $P_{k+1}$-free graph on $n$ vertices, then $e(G) \leq \frac{k-1}{2}n$.
\end{theorem}

Now for our weak base case (which also appears as Corollary 2.1 in \cite{zhli}).

\begin{lemma}\label{bad-bound}
    If $G$ is a $\overrightarrow{P}_{k+1}$-free digraph on $n$ vertices, then 
    \[
    a(G) \leq e(T_{n,k}) + \frac{k-1}{2}n.
    \]
\end{lemma}

\begin{proof}
Beginning with the directed graph $G$, for every bidirected pair of vertices, delete exactly one of the two arcs and let $R$ be the resulting oriented graph.
    Let $B$ be the graph whose edges are pairs of vertices that are bidirected in $G$. Thus, $a(G) = a(R) + e(B)$.

     Since $R$ is a subdigraph of $G$, it is also $\overrightarrow{P}_{k+1}$-free.
     The Tur\'an graph $T_{n,k}$ has the maximum number of edges among $k$-colorable graphs, so Theorem~\ref{GHRV} implies $a(R) \leq e(T_{n,k})$. (Alternatively, we could deduce this from the theorems of R\'edei and T\'uran, as in the introduction.) The graph $B$ is $P_{k+1}$-free, since  otherwise we can find a $\overrightarrow{P}_{k+1}$ among the bidirected pairs in $G$. Therefore, by Theorem~\ref{EG}, we have $e(B) \leq \frac{k-1}{2}n$.
\end{proof}

Let $a(X,Y)$ denote the number of arcs from a set of vertices $X$ to a set of vertices $Y$. If $X$ is a single vertex $x$, we write $a(x,Y)$ instead of $a(\{x\},Y)$. For a digraph $G$ and a set of vertices $X \subseteq V(G)$, let $G-X$ denote the induced subdigraph of $G$ on vertex set $V(G) \setminus X$.

\begin{lemma}\label{main-lemma}
Let $k\geq 2$ and $n\ge 33k^4$, and let $G$ be an extremal $\overrightarrow{P}_{k+1}$-free digraph on $n$ vertices. Then either $G$ contains a directed path $\overrightarrow{P}_k$ on vertex set $P$ such that
\begin{equation}\label{defect-bound}
a(G)-a(G-P) < a(\overrightarrow{T}_{n,k}) - a(\overrightarrow{T}_{n-k,k})
\end{equation}
or $G$ is a transitive Tur\'an digraph $\overrightarrow{T}_{n,k}$.
\end{lemma}

\begin{proof}
As ${G}$ is extremal $\overrightarrow{P}_{k+1}$-free and $n \geq 33k^4$, a simple computation gives 
\[
a(G) \geq a(\overrightarrow{T}_{n,k}) > e(T_{n,k-1}) + \frac{k-2}{2}n.
\]
Now Lemma~\ref{bad-bound} implies that $G$ contains a directed $k$-vertex path $P=x_1x_2\cdots x_{k}$. 
Write  $U=V(G) \setminus P$.

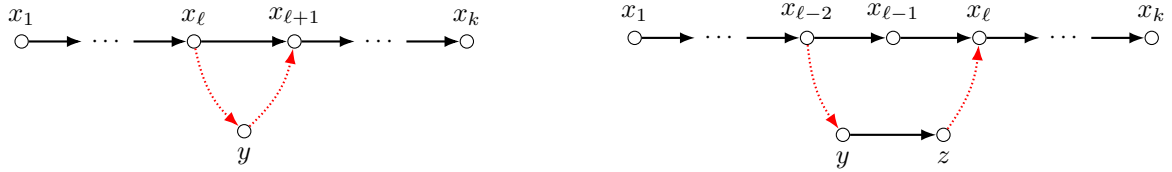
\begin{figure}[ht]
\centering

\begin{minipage}{0.48\textwidth}
\centering
\begin{tikzpicture}[
    scale=0.95,
    v/.style={circle,draw,fill=white,minimum size=5pt,inner sep=0pt},
    arr/.style={-{Latex[length=2mm]},thick},
    bad/.style={-{Latex[length=2mm]},thick, dash pattern=on 0.6pt off 1pt ,red},
    lab/.style={font=\small}
]
    \node[v,label=above:{$x_1$}] (x0) at (0,0) {};
    \node[lab] (d1) at (1.2,0) {$\cdots$};
    \node[v,label=above:{$x_\ell$}] (xi) at (2.4,0) {};
    \node[v,label=above:{$x_{\ell+1}$}] (xip) at (3.8,0) {};
    \node[lab] (d2) at (5,0) {$\cdots$};
    \node[v,label=above:{$x_{k}$}] (xk) at (6.2,0) {};

    \node[v,label=below:{$y$}] (y) at (3.1,-1.25) {};

    \draw[arr] (x0) -- (d1);
    \draw[arr] (d1) -- (xi);
    \draw[arr] (xi) -- (xip);
    \draw[arr] (xip) -- (d2);
    \draw[arr] (d2) -- (xk);

    \draw[bad,bend right=18] (xi) to (y);
    \draw[bad,bend right=18] (y) to (xip);

\end{tikzpicture}
\end{minipage}
\hfill
\begin{minipage}{0.48\textwidth}
\centering
\begin{tikzpicture}[
    scale=0.95,
    v/.style={circle,draw,fill=white,minimum size=5pt,inner sep=0pt},
    arr/.style={-{Latex[length=2mm]},thick},
    bad/.style={-{Latex[length=2mm]},thick,dash pattern=on 0.6pt off 1pt,red},
    lab/.style={font=\small}
]
    \node[v,label=above:{$x_1$}] (x0) at (0,0) {};
    \node[lab] (d1) at (1.2,0) {$\cdots$};
    \node[v,label=above:{$x_{\ell-2}$}] (xlm) at (2.4,0) {};
    \node[v,label=above:{$x_{\ell-1}$}] (xlmone) at (3.6,0) {};
    \node[v,label=above:{$x_\ell$}] (xl) at (4.8,0) {};
    \node[lab] (d2) at (6,0) {$\cdots$};
    \node[v,label=above:{$x_{k}$}] (xk) at (7.2,0) {};

    \node[v,label=below:{$y$}] (y) at (2.9,-1.35) {};
    \node[v,label=below:{$z$}] (z) at (4.3,-1.35) {};

    \draw[arr] (x0) -- (d1);
    \draw[arr] (d1) -- (xlm);
    \draw[arr] (xlm) -- (xlmone);
    \draw[arr] (xlmone) -- (xl);
    \draw[arr] (xl) -- (d2);
    \draw[arr] (d2) -- (xk);

    \draw[bad,bend right=15] (xlm) to (y);
    \draw[arr] (y) -- (z);
    \draw[bad,bend right=15] (z) to (xl);

\end{tikzpicture}
\end{minipage}

\caption{The forbidden configurations used in Claims~\ref{degree-sum} and \ref{trans-classes}. At most one of the red (dotted) arcs is present.
}
\end{figure}

\begin{claim}\label{degree-sum}
 If $y\in U$, then $a(y,P)+a(P,y) \le k-1$.
\end{claim}

\begin{proofc}
      Neither arc $yx_1$ nor $x_{k}y$ can be present in ${G}$, or else ${G}$ would contain a $\overrightarrow{P}_{k+1}$. Of the remaining $2(k-1)$ potential arcs between $y$ and $P$, if, for some $1\le \ell \le k-1$, arc $x_\ell y$ is present, arc $yx_{\ell+1}$ must be absent, or else
${G}$ would again contain a $\overrightarrow{P}_{k+1}$, namely $x_1x_2\cdots x_\ell yx_{\ell+1}\cdots x_{k}$. 
\end{proofc}

 Suppose ${G}[P]$ contains $\binom{k}{2}+c$ arcs,
 so that $c\le\binom{k}{2}$. Let $C \subseteq U$ be the vertices $y \in U$ with 
 $a(y,P)+a(P,y)\le k-2$. 
Therefore, by Claim~\ref{degree-sum}, $a(U,P) +a(P,U) \leq (k-1)(n-k)-|C|$.
Therefore,
\[
a(G) - a(G[U]) = a(G[P]) + a(U,P) + a(P,U) \leq \binom{k}{2} + c +(k-1)(n-k)-|C|.
\]
On the other hand,
\[
a(\overrightarrow{T}_{n,k}) - a(\overrightarrow{T}_{n-k,k}) = \binom{k}{2}+(k-1)(n-k).
\]
Combining these estimates gives
\begin{equation}\label{compare-ineq}
a(G) - a(G[U]) \leq a(\overrightarrow{T}_{n,k}) - a(\overrightarrow{T}_{n-k,k}) +c -|C|.
\end{equation}
If $c<|C|$, then \eqref{defect-bound} is established and we are done.
From here we may assume that $c \geq |C| \geq 0$. 

Partition the vertices of $U' := U \setminus C$ into $U_1 \cup U_2 \cup \cdots \cup U_{k}$ such that $y \in U_i$ if and only if $y$ has $a(P,y)=i-1$ in-neighbors from $P$ and, consequently, by Claim~\ref{degree-sum}, $a(y,P)=k-i$ out-neighbors to $P$.

\begin{claim}\label{trans-classes}
    If $yz$ is an arc with $y \in U_i$ and $z \in U_j$, then $i<j$.
\end{claim}

\begin{proofc}
         Arcs $zx_1$ and $zx_2$ cannot be present, since otherwise ${G}$ would contain a $\overrightarrow{P}_{k+1}$.
         Similarly arcs $x_{k-1}y$ and $x_{k}y$ cannot exist.
         Also, if, for some $3\le \ell\le k$, arc $zx_\ell$ is present, arc $x_{\ell-2}y$ must be absent, or else ${G}$ would again contain a $\overrightarrow{P}_{k+1}$, namely $x_1x_2\cdots x_{\ell-2}yzx_\ell\cdots x_{k}$. 
         Therefore, 
\[
a(z,P)+a(P,y)\le k-2.
\]
         Now, $z \in U_j$ implies $a(z,P) = k-j$, and so
         $i-1 =  a(P,y) \leq (k-2) - (k-j) = j-2$ which gives $i<j$, as desired.
\end{proofc}

Now we have that the partition of $U'$ has a transitive structure. In particular, $G[U']$ is a subdigraph of $H$ where $H$ contains exactly every arc from $U_i$ to $U_j$ when $i<j$. Moreover, $a(H) \leq a(\overrightarrow{T}_{m,k})$ where $m:=|U'| = n-k-|C| \geq 33k^4 - k -k^2 \geq 32k^4$.
As $G$ is extremal and $\overrightarrow{T}_{n,k}$ is $\overrightarrow{P}_{k+1}$-free, $a(G) \geq a(\overrightarrow{T}_{n,k})$. Combining with \eqref{compare-ineq} implies 
\begin{equation*}
a(G[U]) \geq a(\overrightarrow{T}_{n-k,k}) -c \geq  a(\overrightarrow{T}_{n-k,k})-k^2.
\end{equation*}

Removing $C$ from $U$ to obtain $U'$ deletes at most $2|C|m +|C|^2 \leq \frac{3}{2}k^2m$ arcs from $U$.
Therefore,
\[
a(G[U']) + \frac{3}{2} k^2m \geq a(G[U]).
\]
Thus, as $G[U'] \subseteq H$, 
\begin{equation}\label{Uprime}
a(H) \geq a(G[U']) \geq a(\overrightarrow{T}_{n-k,k})-k^2  -\frac{3}{2} k^2m > a(\overrightarrow{T}_{m,k}) - 2k^2m.
\end{equation}

By a standard rebalancing argument we can show that the classes $U_i$ are not too small. Indeed, suppose that some $U_i$ has less than $\frac{m}{2k}$ vertices.
 Moving a vertex from a class of size $t$ to a class of size $s$ where $t \geq s+2$ increases the number of arcs by $t-s-1$. 
As long as there is a class of size $s\leq\frac{3m}{4k}$, then there is a class of size $t \geq \frac{m}{k}+1$. 
 Repeatedly move a vertex from a largest class to $U_i$ until $U_i$ has at least $\frac{3m}{4k}$ vertices.
 Each vertex move increases the number of arcs by at least $\frac{m}{k}+1 -\frac{3m}{4k}-1 \geq \frac{m}{4k}$, and at least $\frac{m}{4k}$ vertices are moved in this way for an increase of at least $\frac{m^2}{16k^2}$ arcs. Therefore, as $\overrightarrow{T}_{m,k}$ has the maximum number of arcs among all such transitive complete $k$-partite digraphs on $m$ vertices,
\[
a(\overrightarrow{T}_{m,k}) - a(H) \geq \frac{m^2}{16k^2}.
\]
However, by \eqref{Uprime}, $a(\overrightarrow{T}_{m,k}) - a(H) < 2k^2m$ which yields $2k^2m > \frac{m^2}{16k^2}$, contradicting $m \geq 32k^4$.
Therefore, $|U_i| \geq \frac{m}{2k}$ for all $i$.

Define $F_i \subseteq U_i$ as follows. Let $F_1 = U_1$ and for $i \geq 2$, let $F_i$ be the vertices $y \in U_i$ such that there exists a directed path $y_1y_2\cdots y_{i-1}y$ with $y_j \in U_j$ for each $j=1,\dots, i-1$.

\begin{claim}\label{F-nonempty}
    $|F_i| \geq \frac{m}{4k}$.
\end{claim}

\begin{proofc}
We proceed by induction on $i$. For $i=1$, this is immediate as 
$|F_1| = |U_1| \geq \frac{m}{2k}$.
Now assume $|F_{i-1}| \geq \frac{m}{4k}$.

For $y \in U_i \setminus F_i$, no vertex of $F_{i-1}$ sends an arc to $y$, otherwise $y \in F_i$. On the other hand, in $H$, every arc from $F_{i-1} \subseteq U_{i-1}$ to $U_i$ is present. Therefore, by \eqref{Uprime},
$|F_{i-1}||U_i \setminus F_i| \leq 2k^2m$. Rearranging and applying induction gives
\[
|U_i \setminus F_i| \leq \frac{2k^2m}{|F_{i-1}|}\leq \frac{2k^2m}{m/(4k)} = 8k^3.
\]
Therefore, using $m \geq 32k^4$, we have
\[
|F_i| = |U_i| - |U_i \setminus F_i| \geq \frac{m}{2k} - 8k^3 \geq \frac{m}{4k}.
\]
\end{proofc}

As a consequence of Claim~\ref{F-nonempty}, each $F_i$ is non-empty.
Now, consider a vertex $y \in F_i$. By the definition of $F_i$, there is a path $y_1y_2\cdots y_{i-1}y$ in $U$. If $yx_j$ is an arc for some $j \leq i$, then $y_1y_2\cdots y_{i-1}yx_jx_{j+1}\cdots x_{k}$ is a directed path with $i+1 + (k-1-j) = k+i-j \geq k$ arcs, a contradiction. Therefore, the $k-i$ out-neighbors of $y$ in $P$ are exactly $\{x_{i+1},\dots, x_{k}\}$. Now no $x_\ell y$ for $\ell=i,\dots, k-1$ can be an arc since this would create a $\overrightarrow{P}_{k+1}$, so the $i-1$ in-neighbors of $y$ in $P$ are exactly $\{x_1,\dots, x_{i-1}\}$ for $i\geq 2$.

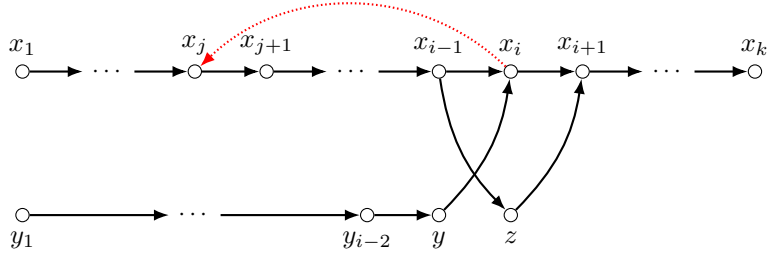
\begin{figure}[ht]
\centering

\centering
\begin{tikzpicture}[
    scale=0.95,
    v/.style={circle,draw,fill=white,minimum size=5pt,inner sep=0pt},
    arr/.style={-{Latex[length=2mm]},thick},
    bad/.style={-{Latex[length=2mm]},thick,red},
    lab/.style={font=\small}
]
    \node[v,label=above:{$x_1$}] (x0) at (0,2) {};
    \node[lab] (d1) at (1.2,2) {$\cdots$};
    \node[v,label=above:{$x_j$}] (xj) at (2.4,2) {};
    \node[v,label=above:{$x_{j+1}$}] (xjp1) at (3.4,2) {};
    \node[lab] (d2) at (4.6,2) {$\cdots$};
    \node[v,label=above:{$x_{i-1}$}] (xim1) at (5.8,2) {};
    \node[v,label=above:{$x_i$}] (xi) at (6.8,2) {};
    \node[v,label=above:{$x_{i+1}$}] (xip1) at (7.8,2) {};
    \node[lab] (d3) at (9,2) {$\cdots$};
    \node[v,label=above:{$x_{k}$}] (xk) at (10.2,2) {};

    \node[v,label=below:{$y_1$}] (y0) at (0,0) {};
    \node[lab] (e1) at (2.4,0) {$\cdots$};
    \node[v,label=below:{$y_{i-2}$}] (yim2) at (4.8,0) {};
    \node[v,label=below:{$y$}] (y) at (5.8,0) {};

    \node[v,label=below:{$z$}] (z) at (6.8,0) {};

    \draw[arr] (x0) -- (d1);
    \draw[arr] (d1) -- (xj);
    \draw[arr] (xj) -- (xjp1);
    \draw[arr] (xjp1) -- (d2);
    \draw[arr] (d2) -- (xim1);
    \draw[arr] (xim1) -- (xi);
    \draw[arr] (xi) -- (xip1);
    \draw[arr] (xip1) -- (d3);
    \draw[arr] (d3) -- (xk);

    \draw[arr] (y0) -- (e1);
    \draw[arr] (e1) -- (yim2);
    \draw[arr] (yim2) -- (y);

\draw[arr,bend right=18] (y) to (xi);
    \draw[arr,bend right=18]  (xim1) to (z);
    \draw[arr,bend right=18]  (z) to (xip1);
    
    \draw[bad,dash pattern=on 0.6pt off 1pt,bend right=44] (xi) to (xj);

\end{tikzpicture}

\caption{A directed path containing a $\overrightarrow{P}_{k+1}$, using arc $x_ix_j$ in the proof of Claim~\ref{no-back}.}\label{back-fig}
\end{figure}

We now rule out ``backward'' arcs in $P$.

\begin{claim}\label{no-back}
    There is no arc $x_ix_j$ with $i>j$ in $P$.
\end{claim}

\begin{proofc}
    Suppose such an arc $x_ix_j$ is present (see Figure~\ref{back-fig}). Choose $y \in F_{i-1}$ and $z \in F_i$. By definition,
there is a directed path $y_1y_2\dots y_{i-2}y$ in $U$. Moreover, there are arcs $yx_i$ and $x_{i-1}z$, and, if $i \leq k-1$, there is also arc $zx_{i+1}$. 
Now consider the directed path \[
y_1y_2\dots y_{i-2}yx_ix_jx_{j+1}\cdots x_{i-1}zx_{i+1}\cdots x_{k}
\] 
where the final segment of the path after $z$ is empty if $i=k$. This path has $(i-1)+1+(i-j)+1+(k-i)=k+i-j+1\geq k+1$ vertices, a contradiction.
\end{proofc}

By Claim~\ref{no-back}, $a(G[P]) \leq \binom{k}{2}$ and thus $c \leq 0$. Since we are in the case $c \geq |C| \geq 0$, it follows that $c=|C|=0$ and $a(G[P]) = \binom{k}{2}$. 
Therefore, we obtain both that $G[P]$ is a transitive tournament and that $U'=U$. 
By \eqref{compare-ineq}, we have $a(G) - a(G[U]) \leq a(\overrightarrow{T}_{n,k}) - a(\overrightarrow{T}_{n-k,k})$. Since $G$ is extremal, $a(G) \geq a(\overrightarrow{T}_{n,k})$ and so $a(G[U]) \geq a(\overrightarrow{T}_{n-k,k})$. Moreover, the discussion after Claim~\ref{trans-classes} implies that $G[U]$ is a subdigraph of the transitive complete $k$-partite $H$. Therefore, $G[U] = H = \overrightarrow{T}_{n-k,k}$.
Now it is straightforward to check, using the definition of the classes $U_i$ and that $G$ is $\overrightarrow{P}_{k+1}$-free, that for all $i$, we may put $V_i := U_i \cup \{x_i\}$ and obtain that $G$ is a transitive Tur\'an digraph $\overrightarrow{T}_{n,k}$ with classes $V_1,\dots, V_k$.
\end{proof}

We are now ready to prove the main theorem.

\begin{proof}[Proof of Theorem~\ref{main-thm}.]
The case $k=1$ is trivial, so assume $k \geq 2$.
    Let $\mathrm{EX}(n,\overrightarrow{P}_{k+1})$ be the family of extremal digraphs for $\overrightarrow{P}_{k+1}$.
    For $G \in \mathrm{EX}(n,\overrightarrow{P}_{k+1})$, define
    \[
    f(G) :=  a(G) - a(\overrightarrow{T}_{n,k}).
    \]
    Since $\overrightarrow{T}_{n,k}$ is $\overrightarrow{P}_{k+1}$-free, we have $f(G) \geq 0$ for every $G \in \mathrm{EX}(n,\overrightarrow{P}_{k+1})$.

    Set $N=33k^4$ and let $m \geq N$.
    If $H \in \EX(m,\overrightarrow{P}_{k+1})$ is not a transitive Tur\'an digraph $\overrightarrow{T}_{m,k}$, then
    Lemma~\ref{main-lemma} implies that $H$ contains a directed path $\overrightarrow{P}_{k}$ on vertex set $P$ such that
    \[
    a(H)-a(H-P) < a(\overrightarrow{T}_{m,k}) - a(\overrightarrow{T}_{m-k,k}).
    \]
    Let $H' \in \mathrm{EX}(m-k,\overrightarrow{P}_{k+1})$. Since $H-P$ is still $\overrightarrow{P}_{k+1}$-free, we have $a(H') \geq a(H-P)$.
    Therefore,
    \begin{align*}
        f(H') & = a(H')-a(\overrightarrow{T}_{m-k,k}) \\
        &\geq a(H-P) -a(\overrightarrow{T}_{m-k,k}) \\
        & > a(H)  -a(\overrightarrow{T}_{m,k}) \\
        &=f(H). 
    \end{align*}
    In particular, $f(H') \geq f(H)+1$ and, moreover, $H'$ cannot be a transitive Tur\'an digraph $\overrightarrow{T}_{m-k,k}$.

Now consider all extremal digraphs on $m$ vertices where $N \leq m < N+k$. 
Let $M$ be the maximum value of $f(H)$ over all such digraphs $H$, i.e.,
\[
M:= \max \left\{f(H) \mid H \in \EX(m,\overrightarrow{P}_{k+1}) \text{ for } N \leq m < N+k \right\}.
\]
If $H$ has $m$ vertices, then Lemma~\ref{bad-bound} implies $f(H) = a(H) - a(\overrightarrow{T}_{m,k}) = a(H) - e(T_{m,k}) \leq \frac{k}{2}m$. Therefore, $M <  \frac{k}{2}(N+k) \leq kN \leq 33k^5$.

Now, suppose for a contradiction, that for $n \geq 50k^6$ there is a digraph $G \in \EX(n,\overrightarrow{P}_{k+1})$ that is not a transitive Tur\'an digraph $\overrightarrow{T}_{n,k}$.
Repeatedly apply the step above by replacing an extremal digraph $G$ with an extremal digraph $G'$ on $k$ fewer vertices until $G'$ has at least $N$ and fewer than $N+k$ vertices. At each step $f(\cdot)$ increases by at least $1$ and the number of steps is at least $\lfloor \frac{n-N}{k} \rfloor \geq 50k^5-33k^3-1 > 33k^5 \geq M$ for $k\geq 2$. This yields $G'$ on at least $N$ and less than $N+k$ vertices
with $f(G') > M$, a contradiction.
Therefore, every extremal digraph on $n \geq 50k^6$ vertices must be a transitive Tur\'an digraph $\overrightarrow{T}_{n,k}$.
\end{proof}

\bibliography{ref}
\bibliographystyle{plainurl}

\end{document}